\documentclass[12pt]{amsart}
\usepackage{amsrefs}
\usepackage[dvipdfm, bookmarkstype=toc=true, linktocpage=true, bookmarks=true]{hyperref}
\usepackage{mathptmx}
\usepackage{color}
\usepackage{amssymb}
\usepackage[all]{xy}
\newif\ifdebug

\debugfalse

\newcommand{\printname}[1]{\smash{\makebox[0pt]{\hspace{-1.0in}\raisebox{8pt}{\tiny #1}}}}

\newcommand{\Label}[1]{\ifdebug{\label{#1}\printname{#1}}\else{\label{#1}}\fi}

\theoremstyle{plain}
\newtheorem{theo}{Theorem}[section]
\newtheorem{lemm}[theo]{Lemma}
\newtheorem{coro}[theo]{Corollary}
\newtheorem{prop}[theo]{Proposition}
\newtheorem{conj}[theo]{Conjecture}

\theoremstyle{remark}
\newtheorem{rema}[theo]{Remark}

\theoremstyle{definition}

\newtheorem{exam}[theo]{Example}

\def\C{\mathbb C}
\def\Z{\mathbb Z}
\def\R{\mathbb R}

\def\F{\mathcal F}
\def\g{\mathfrak{g}}

\def\id{\mathrm{id}}

\def\pos{\mathrm{pos}}

\def\co{\thinspace \colon}
\def\Hom{\mathrm{Hom}}

  \makeatletter
    
    \@addtoreset{equation}{section}
  \makeatother

\begin{document}
\title[Torus invariant transverse K\"{a}hler foliations]{Torus invariant transverse K\"{a}hler foliations}

\author[H.~Ishida]{Hiroaki Ishida}
\address{Department of Mathematics and Computer Science, Graduate School of Science and Engineering, Kagoshima University}
\email{ishida@sci.kagoshima-u.ac.jp}

\date{\today}
\thanks{The author was supported by JSPS Research Fellowships for Young 
Scientists}

\keywords{Torus action, complex manifold, toric variety, moment-angle manifold, LVM manifold, LVMB manifold, non-K\"{a}hler manifold, transverse K\"{a}hler form, moment map}
\subjclass[2010]{Primary 53D20, Secondary 14M25, 32M05, 57S25}

\begin{abstract}
	In this paper, we show the convexity of the image of a moment map on a transverse symplectic manifold equipped with a torus action under a certain condition. We also study properties of moment maps in the case of transverse K\"{a}hler manifolds. As an application, we give a positive answer to the conjecture posed by Cupit-Foutou and Zaffran. 
\end{abstract}

\maketitle
\section{Introduction}\Label{sec:intro}
	In \cite{Lopez-Verjovsky}, a family of complex manifolds (say LV manifolds here) which includes classical Hopf manifolds and Calabi-Eckmann manifolds are constructed by L\'{o}pez de Medrano and Verjovsky. Most of LV manifolds are non-K\"{a}hler; it is shown that an LV manifold admits a K\"{a}hler form if and only if it is a compact complex torus of complex dimension $1$. In \cite{Loeb-Nicolau}, as a contrast, it is shown by Loeb and Nicolau that each LV manifold carries a transverse K\"{a}hler vector field. In \cite{Meersseman}, Meersseman generalizes the construction of LV manifolds and gives a new family of complex manifolds which are known as LVM manifolds. As well as LV manifolds, an LVM manifold admits a K\"{a}hler form if and only if it is a compact complex torus. He also constructs a foliation $\F$ on each LVM manifold and shows that $\F$ is transverse K\"{a}hler. In \cite{Bosio}, Bosio generalizes LVM manifolds and now they are known as LVMB manifolds. As well as LVM manifolds, if an LVMB manifold is K\"{a}hler then it is a compact complex torus.  The first example of LVMB manifold which is not biholomorphic to any LVM manifold is given by Cupit-Foutou and Zaffran in \cite{Foutou-Zaffran}. In particular, the family of LVMB manifolds properly contains the family of LVM manifolds. In \cite{Battisti}, Battisti gives an explanation of the difference between LVM manifolds and LVMB manifolds in terms of toric geometry.

	In \cite{Foutou-Zaffran}, Cupit-Foutou and Zaffran also construct a foliation $\F$ on each LVMB manifold, as a natural generalization of the case of LVM manifolds. They show that, under an assumption, if the foliation $\F$ on an LVMB manifold $M$ is transverse K\"{a}hler then $M$ is an LVM manifold. From this point of view, they give the following conjecture which is the motivation of this paper: 
	\begin{conj}\Label{conj:FZ}
		An LVMB manifold is an LVM manifold if and only if the foliation $\F$ is transverse K\"{a}hler. 
	\end{conj}
	Our approach to Conjecture \ref{conj:FZ} uses techniques of Hamiltonian torus actions on symplectic manifolds. Especially, (an analogue of) the convexity theorem plays an important role. 
	The convexity theorem shown by Atiyah, Guillemin and Sternberg in \cite{Atiyah} and \cite{Guillemin-Sternberg}  states that if a compact torus $G$ acts on a compact connected symplectic manifold $M$ in Hamiltonian fashion then the image of a corresponding moment map is a convex polytope. In this paper, first we show the following:
	\begin{theo}[see also Theorem \ref{theo:convexity}]\Label{theo:introconvexity}
		Let $M$ be a compact connected manifold equipped with an action of a compact torus $G$. Let $\g'$ be a subspace of the Lie algebra $\g$ of $G$ such that the action of $\g'$ is local free. Let $\F_{\g'}$ be the foliation on $M$ whose leaves are $\g'$-orbits. Let $\omega$ be a $G$-invariant transverse symplectic form on $M$ with respect to $\F_{\g'}$. If there exists a moment map $\Phi \co M \to \g^*$ with respect to $\omega$, then the image of $M$ by $\Phi$ is the convex hull of the image of common critical points of $h_v$ for $v \in \g$, where $h_v \co M \to \R$ is given by $h_v (x)= \langle \Phi (x), v\rangle$. 
	\end{theo}
	\begin{exam}
		The setting of Theorem \ref{theo:introconvexity} naturally appears in symplectic manifolds. 
		Let $N$ be a compact connected symplectic manifold equipped with an effective Hamiltonian action of a compact torus $G$. Let $\g'$ be any subspace of $\g$ and let $i\co \g' \to \g$ denote the inclusion. Let $c \in (\g')^*$ be a regular value of $i^* \circ \Phi \co N \to (\g')^*$ and let $M := (i^* \circ \Phi)^{-1}(c)$. Then, $\g'$ acts on $M$ local freely and the restriction $\omega|_M$ of the symplectic form on $N$ is a transverse symplectic form on $M$ with respect to $\F_{\g'}$. The image $\Phi (M)$ coincides with $(i^*)^{-1}(c) \cap \Phi (N)$, and it is a convex polytope. 
	\end{exam}
	
	For a connected complex manifold $M$ equipped with an effective action of a compact torus $G$ which preserves the complex structure $J$ on $M$, the subspace 
	\begin{equation*}
		\g_J := \{ v \in \g \mid X_v = -JX_{v'}, ^\exists v' \in \g\}
	\end{equation*}
	of $\g$ acts on $M$ local freely (see Proposition \ref{prop:localfree}). The foliation $\F_{\g_J}$ whose leaves are $\g_J$-orbits coincides with the foliation which has been considered in \cite{Foutou-Zaffran} in case of LVMB manifolds and the foliation which has been considered in \cite{Panov-Ustinovsky-Verbitsky} in case of moment-angle manifolds equipped with complex-analytic structures. $\F_{\g_J}$ gives a lower bound of $G$-invariant foliations that admit $G$-invariant transverse K\"{a}hler forms such that there exist moment maps with respect to the forms (see Proposition \ref{prop:lowerbound}). 
	
	An effective action of a compact torus $G$ on a connected manifold $M$ is said to be \emph{maximal} if there exists a point $x \in M$ such that $\dim G + \dim G_x = \dim M$ (see \cite{Ishida} for detail). If $M$ is a compact connected complex manifold equipped with a maximal action of a compact torus $G$ preserving the complex structure $J$ on $M$, one can associate a complete fan $q(\Delta)$ in $\g/\g_J$ with $M$. On the other hand, if there exists a $G$-invariant transverse K\"{a}hler form on $M$ with respect to $\F_{\g_J}$ and if there exists a moment map $\Phi \co M \to \g^*$ with respect to the form, then one can find a lift $\widetilde{\Phi} \co M \to (\g/\g_J)^*$ of $\Phi$. The following is the main theorem in this paper:
	\begin{theo}[see also Theorem \ref{theo:polytopal}]\Label{theo:intropolytopal}
		Let $M$ be a compact connected complex manifold equipped with a maximal action of a compact torus $G$ which preserves the complex structure $J$ on $M$. If $\F_{\g_J}$ is transverse K\"{a}hler, then a moment map $\Phi$ with respect to a $G$-invariant transverse K\"{a}hler form exists and $\widetilde{\Phi}(M)$ is a convex polytope normal to $q(\Delta)$. Conversely, if $q(\Delta)$ is polytopal, then $\F_{\g_J}$ is transverse K\"{a}hler. 
	\end{theo}
	
	As an application of Theorem \ref{theo:intropolytopal}, we show that Conjecture \ref{conj:FZ} is true. 
		
	This paper is organized as follows. In Section \ref{sec:convexity}, we investigate Hamiltonian functions for almost periodic vector fields on transverse symplectic manifold and show the convexity of the image of a moment map with an almost same argument as Atiyah. In Section \ref{sec:holofoli}, we construct a $G$-invariant foliation $\F_{\g_J}$ on a complex manifold equipped with an action of a compact torus $G$. In Section \ref{sec:invtransverse}, we show that the foliation $\F_{\g_J}$ is a lower bound of $G$-invariant foliations that  admit moment maps. In Section \ref{sec:extreme}, we consider the case of maximal torus action and give a proof of the conjecture posed by Cupit-Foutou and Zaffran. 
	\smallskip
	
	\noindent {\bf Convention and notation.} For a smooth manifold $M$ and a smooth foliation $\F$ on $M$, we denote by $T\F$ the subbundle of the tangent bundle $TM$ consisting of vectors tangent to leaves of $\F$. Let $G$ be a compact torus acting on $M$ smoothly. We say that $\F$ is \emph{$G$-invariant} if $T\F$ is a $G$-equivariant subbundle of $TM$. We denote by $\g$ the Lie algebra of $G$. Through the exponential map, $\g$ also acts on $M$. For a subspace $\g'$ that acts on $M$ local freely, we denote by $\F_{\g'}$ the foliation on $M$ whose leaves are $\g'$-orbits. For a point $x \in M$, we denote by $G_x$ the isotropy subgroup of $G$ at $x$. We also denote by $\g_x$ the Lie algebra of $G_x$. Remark that $\g_x$ is not the isotropy subgroup of $\g$ at $x$. For $v \in \g$, we denote by $X_v$ the fundamental vector field generated by $v$ on $M$. For a vector field $X$ on $M$, we denote by $X_x$ the value of $X$ at $x$. For the fundamental vector field, we denote by $(X_v)_x$ the value of $X_v$ at $x$. For a differential form $\omega$, we denote by $\omega_x$ the value of $\omega$ at $x$. We denote by $\mathcal{L}_X\omega$ the Lie derivative for $\omega$ along $X$ and by $\iota_X\omega$ the interior product of $X$ and $\omega$. We identify $\R$ with the Lie algebra of $S^1$ by the differential of the map $t \mapsto e^{2\pi\sqrt{-1}t}$. For $\alpha \co G \to S^1$, we denote by $d\alpha$ the differential at the unit of $G$ and $d\alpha$ is regarded as an element in $\g^*$. We denote by $H_\F^1(M)$ the first basic cohomology group with coefficients in $\R$.

\section{The convexity theorem}\Label{sec:convexity}
	Let $M$ be a smooth manifold and let $\F$ be a smooth foliation on $M$. A \emph{transverse symplectic form} $\omega$ with respect to $\F$ is a closed $2$-form on $M$ whose kernel coincides with $T\F$. 
	Let $\omega$ be a transverse symplectic form on $M$ with respect to $\F$. Let a compact torus $G$ act on $M$ effectively. We assume that the action of $G$ preserves $\omega$ (and hence, $\F$ is $G$-invariant). In this case, by Cartan formula we have that 
	\begin{equation*}
		0= \mathcal{L}_{X_v}\omega = d\iota_{X_v}\omega +\iota_{X_v}d\omega
		= d\iota_{X_v}\omega
	\end{equation*}
	for $v \in \mathfrak{g}$. We say that a smooth map $\Phi \co M \to \mathfrak{g}^*$ is a \emph{moment map} if the function $h_v \co M \to \R$ given by $\langle \Phi (x), v\rangle = h_v(x)$ satisfies that $dh_v = -\iota_{X_{v}}\omega$. A moment map $\Phi$ with respect to $\omega$ exists if and only if $\iota_{X_v}\omega$ is exact for any $v$. In particular, the obstruction for the existence of moment map sits in $H^1_\F(M)$. 
	The purpose in this section is to show the convexity of the image of a moment map under certain conditions by an almost same argument as Atiyah (see \cite{Atiyah}). 
	
	Let $\{\psi_\alpha \co U_\alpha \to V_\alpha\}_{\alpha}$ be foliation charts of $(M, \F)$. The local leaf space $U_\alpha/\F$ is diffeomorphic to an open subset of $\R^{\dim M -\dim \F}$. The quotient map $\pi_\alpha \co U_\alpha \to U_\alpha/\F$ is a fiber bundle whose fibers are diffeomorphic to open balls of dimension $\dim \F$. The transition functions $\psi_{\alpha\beta} \co \psi_\beta(U_\alpha \cap U_\beta) \to \psi_\alpha(U_\alpha\cap U_\beta)$ can be written as 
	\begin{equation*}
		\psi_{\alpha\beta}(x,z) = (\psi_{\alpha\beta}^T(x), \psi_{\alpha\beta}^\F(x,z)) \in \psi_{\alpha}(U_\alpha\cap U_\beta) \subseteq \R^{\dim M -\dim \F} \times \R^{\dim \F}
	\end{equation*}
	for $(x, z) \in \psi_{\beta}(U_\alpha \cap U_\beta) \subseteq \R^{\dim M -\dim \F } \times \R^{\dim \F}$. 
	
	Let $x \in M$ and let $G_x$ denote the isotropy subgroup at $x$ of $G$. Let $\psi \co U \to V$ be a local foliation chart on an open neighborhood at $x$. The local leaf space $U/\F$ is diffeomorphic to an open subset of $\R^{\dim M - \dim \F}$. Let $\pi \co U \to U/\F$ be the quotient map. Since $G_x$ is compact, the intersection $U' := \bigcap_{g \in G_x} g(U) $ is a $G_x$-invariant open neighborhood at $x$. Since $\F$ is $G$-invariant, $\pi (U')$ is a $G_x$-manifold of dimension $\dim M - \dim \F$. Let $\omega$ be a $G$-invariant transverse symplectic form on $M$ with respect to $\F$ and suppose that $\dim \F = \ell$ and $\dim M = 2n + \ell$. $\omega$ descends to a symplectic form $\underline{\omega}$ on $\pi (U')$. By equivariant Darboux theorem, there exist $G_x$-invariant open subsets $U_x$ of $\pi (U')$ and $V_x$ of $T_xM/T_x\F$, a $G_x$-equivariant diffeomorphism $\varphi \co U_x \to V_x$ and a basis $(x_1,\dots, x_n, y_1,\dots, y_n)$ of $(T_xM/T_x\F)^*$ such that 
	\begin{equation*}
		(\varphi^{-1})^*\underline{\omega} = \sum_{i=1}^ndx_i \wedge dy_i
	\end{equation*}
	and 
	\begin{equation*}
		X_v = 2\pi\sum_{i=1}^n \left\langle d\alpha_i, v \right\rangle \left( x_i\frac{\partial}{\partial y_i}-y_i\frac{\partial}{\partial x_i}\right)
	\end{equation*} 
	for $v \in \mathfrak{g}_x$, where $\alpha_1,\dots, \alpha_n \in \Hom (G_x, S^1)$ are weights at $0 \in T_xM/T_x\F$.
	Remark that $\psi|_{\pi^{-1}(U_x)} \co \pi^{-1}(U_x) \to \psi(\pi^{-1}(U_x))$ is a local foliation chart near $x$.  
	We state this fact as a lemma for later use. 
	\begin{lemm}\Label{lemm:eDarboux}
		Let $M$ be a smooth manifold of dimension $2n+\ell$ equipped with an action of a compact torus $G$. Let $\F$ be a $G$-invariant smooth foliation on $M$ of dimension $\ell$ and let $\omega$ be a $G$-invariant transverse symplectic form on $M$ with respcet to $\F$. Then, for any $x \in M$, there exist 
		\begin{itemize}
			\item a local foliation chart $\psi_x \co \widetilde{U_x} \to \widetilde{V_x}$ on an open neighborhood $U_x$ at $x$ such that $\widetilde{U_x}/\F$ carries the action of $G_x$,  
			\item a $G_x$-invariant open neighborhood $V_x$ at $0$ of $T_xM/T_x\F$, 
			\item a $G_x$-equivariant diffeomorphism $\varphi_x \co \widetilde{U_x}/\F \to V_x$, and 
			\item a basis $(x_1,\dots, x_n, y_1,\dots, y_n)$ of $(T_xM/T_x\F)^*$ 
		\end{itemize}
		such that $(\varphi_x^{-1})^*\underline{\omega} = \sum_{i=1}^ndx_i \wedge dy_i$ and 
		\begin{equation*}
			X_v = 2\pi\sum_{i=1}^n \left\langle d\alpha_i, v \right\rangle \left( x_i\frac{\partial}{\partial y_i}-y_i\frac{\partial}{\partial x_i}\right)
		\end{equation*} 
		for $v \in \mathfrak{g}_x$, where $\alpha_1,\dots, \alpha_n \in \Hom (G_x, S^1)$ are weights at $0 \in T_xM/T_x\F$. 
	\end{lemm}
	Let $M, G, \F, \omega$ be as Lemma \ref{lemm:eDarboux}. Let $v \in \mathfrak{g}$ and suppose that there exists a smooth function $h_v \co M \to \R$ such that $-\iota_{X_v}\omega = dh_v$. Since $\omega$ is transverse symplectic with respect to $\F$, a point $x \in M$ is a critical point of $h_v$ if and only if $(X_v)_x \in T_x\F$. Because of this, unfortunately, we can not deduce the property that $h_v$ is non-degenerate for general $\F$, not like symplectic case. Let $\mathfrak{g}'$ be a subspace of $\mathfrak{g}$ such that $\mathfrak{g}'$ act on $M$ local freely. Since $G$ is abelian, $\F_{\mathfrak{g}'}$ is a $G$-invariant foliation. 
	\begin{lemm}\Label{lemm:Morse}
		Let $M$ be a smooth manifold equipped with an action of a compact torus $G$. Let $\mathfrak{g}'$ be a subspace of $\mathfrak{g}$ such that $\mathfrak{g}'$ acts on $M$ local freely. Let $\omega$ be a $G$-invariant transverse symplectic form on $M$ with respect to $\F_{\mathfrak{g}'}$. Let $v \in \mathfrak{g}$ and suppose that there exists a smooth function $h_v \co M \to \R$ such that $dh_v = -\iota_{X_v}\omega$. Then, $h_v$ is a non-degenerate function and the index of each critical submanifold is even. 
	\end{lemm}
	\begin{proof}
		Let $x \in M$ be a critical point of $h_v$. Then, $(X_v)_x \in T_x\F_{\mathfrak{g}'}$ implies that there exist $v _x \in \mathfrak{g}_x$ and $v' \in \mathfrak{g}'$ such that $v = v_x + v'$. Since $\iota_{X_{v'}}\omega =0$, we have that $\iota_{X_v}\omega = \iota_{X_{v_x}}\omega$. Since $\iota_{X_{v}}\omega$ is basic for $\F_{\mathfrak{g}'}$, so is $h_v$. Let $\psi_x \co \widetilde{U_x} \to \widetilde{V_x}$, $V_x$, $\varphi_x$, $(x_1,\dots, x_n, y_1,\dots ,y_n)$ be as Lemma \ref{lemm:eDarboux}. Since $h_v$ is basic for $\F_{\g'}$, $h_v$ descends to a smooth function $\underline{h_v} \co \widetilde{U_x}/\F_{\mathfrak{g}'} \to \R$ such that $\pi^*\underline{h_v} = h_v$, where $\pi \co \widetilde{U_x} \to \widetilde{U_x}/\F_{\g'}$ denote the quotient map. By definition of $\F_{\mathfrak{g}'}$, $\pi$ sends the fundamental vector field $X_v$ generated by $v$ on $\widetilde{U_x}$ to the fundamental vector field $X_{v_x}$ generated by $v_x$ on $\widetilde{U_x}/\F_{\mathfrak{g}'}$. Also, since $\varphi_x$ is $G_x$-equivariant, $\varphi_x$ sends $X_{v_x}$ on $\widetilde{U_x}/\F_{\g'}$ to $X_{v_x}$ on $T_xM/T_x\F_{\g'}$. 
		
		Therefore 
		\begin{equation*}
			dh_v = -\iota_{X_v}\omega = -\iota_{X_{v_x}}\omega = \pi^* (-\iota_{X_{v_x}}\underline{\omega}) = \pi^*\circ \varphi_x^*(-\iota_{X_{v_x}}(\varphi_x^{-1})^*\underline{\omega}). 
		\end{equation*}
		On the other hand, 
		\begin{equation*}
			dh_v = d(\pi^*\underline{h_v}) = \pi^*(d\underline{h_v}) = \pi^*\circ (\varphi_x^{-1})^*(d((\varphi_x^{-1})^*\underline{h_v})). 
		\end{equation*}
		Since $\pi^*$ is injective and $\varphi_x$ is a diffeomorphism, we have that $d((\varphi_x^{-1})^*\underline{h_v}) = -\iota_{X_{v_x}}(\varphi_x^{-1})^*\underline{\omega}$. Let $\alpha_1,\dots, \alpha_n \in \Hom (G_x, S^1)$ be the weights at the origin in $T_xM/T_x\F_{\g'}$. Then, $X_{v_x}$ on $T_xM/T_x\F_{\g'}$ can be represented as 
		\begin{equation*}
			X_{v_x} = 2\pi\sum_{i=1}^n \langle d\alpha_i, v_x\rangle \left(x_i \frac{\partial}{\partial y_i} -y_i\frac{\partial}{\partial x_i}\right)
		\end{equation*}
		with the coordinates $(x_1,\dots, x_n, y_1,\dots, y_n)$. Therefore
		\begin{equation*}
			 -\iota_{X_{v_x}}(\varphi_x^{-1})^*\underline{\omega} = 2\pi\sum_{i=1}^n \langle d\alpha_i, v_x\rangle (x_idx_i+y_idy_i)
		\end{equation*}
		and hence 
		\begin{equation}\Label{eq:index}
			(\varphi_x^{-1})^*\underline{h_v} = (\varphi_x^{-1})^*\underline{h_v}(0)+ \pi \sum_{i=1}^n {\langle d\alpha_i, v_x\rangle}(x_i^2+y_i^2).
		\end{equation}
		Therefore $(\varphi_x^{-1})^*\underline{h_v}$ is nondegenerate at $0$ and the index at $0$ is twice as many as the number of $\alpha_i$ such that $\langle d\alpha_i, v_x\rangle <0$. 
		Since $\varphi_x \circ \pi \co \widetilde{U_x} \to V_x$ is a fiber bundle, $h_v$ is nondegenerate at $x$ and the index at $x$ is twice as many as the number of $\alpha_i$ such that $\langle d\alpha_i, v_x\rangle <0$, proving the lemma. 
	\end{proof}
	\begin{rema}\Label{rema:localminimum}
		In the proof of Lemma \ref{lemm:Morse}, it follows from \eqref{eq:index} that $x$ attains a local minimum of $h_v$ if and only if $\langle d\alpha_i, v_x\rangle \geq 0$ for all $\alpha_i$. 
	\end{rema}
	\begin{rema}
		We are not sure whether Lemma \ref{lemm:Morse} holds even if we replace $\F_{\g'}$ to any $G$-invariant foliation $\F$ or not. 
	\end{rema}
	
	The following is the key of the convexity theorem. 
	\begin{lemm}[{\cite[Lemma 2.1]{Atiyah}}]\Label{lemm:levelset}
		Let $\phi \co N \to \R$ be a non-degenerate function (in the sense of Bott) on the compact connected manifold $N$, and assume that neither $\phi$ or $-\phi$ has a critical manifold of index $1$. Then $\phi^{-1}(c)$ is connected (or empty) for every $c \in \R$. 
	\end{lemm}
	
	By Lemmas \ref{lemm:Morse} and \ref{lemm:levelset}, if $M$ is compact then the level set $h_v^{-1}(c)$ is connected unless empty. Moreover, if $c$ is a regular value then $h_v^{-1}(c)$ is a connected submanifold of $M$. Since $G$ is abelian, $dh_v = -\iota_{X_v}\omega$ is $G$-invariant. Therefore $h_v$ is also $G$-invariant. Therefore $h_v^{-1}(c)$ is $G$-invariant for $c \in \R$. 
	
	\begin{lemm}\Label{lemm:reduction}
		Let $M$ be a compact connected manifold equipped with an action of a compact torus $G$. Let $\g'$ be a subspace of $\g$ such that $\g'$ acts on $M$ local freely. Let $\omega$ be a $G$-invariant transverse symplectic form on $M$ with respect to $\F_{\g'}$. Let $v_1,\dots, v_k \in \g$ and suppose that there exists a smooth function $h_{v_i}$ such that $dh_{v_i} = -\iota_{X_{v_i}}\omega$ for $i=1,\dots, k$. Let $c \in \R^k$ be a regular value of $h = (h_{v_1},\dots, h_{v_k}) \co M \to \R^k$. Then, $\g'' = \g' + \R v_1+\dots +\R v_k$ acts on $h^{-1}(c)$ local freely and $\omega|_{h^{-1}(c)}$ is a transverse symplectic form with respect to the foliation $\F_{\g''}$.  
	\end{lemm}
	\begin{proof}
		Let $x \in h^{-1}(c)$. Since $c$ is a regular value, $(-\iota_{X_{v_i}}\omega)_x$ is linearly independent for all $i$. This together with that the action of $\g'$ is local free yields that $(X_{v''})_x =0$ if and only if $v'' =0$ for $v'' \in \g''$. Therefore the action of $\g''$ is local free. 
		
		$T_x(h^{-1}(c))$ is given by $\ker (dh)_x = (T_x\F_{\g''})^\perp$, where $(T_x\F_{\g''})^\perp$ denotes the annihilator of $T_x\F_{\g''}$ with respect to $\omega$. Therefore $\omega_x$ descends to a symplectic form on $T_x(h^{-1}(c))/T_x\F_{\g''}$. It turns out that $Y_x \in \ker (\omega|_{h^{-1}(c)})_x$ if and only if $Y_x \in T_x\F_{\g''}$. Therefore $\omega|_{h^{-1}(c)}$ is a transverse symplectic form with respect to $\F_{\g''}$, proving the lemma.
	\end{proof}
	Now we are in a position to prove the convexity theorem. 
	\begin{theo}\Label{theo:convexity}
		Let $M$ be a compact connected manifold equipped with an action of a compact torus $G$. Let $\g'$ be a subspace of $\g$ such that $\g'$ acts on $M$ local freely. Let $\omega$ be a $G$-invariant transverse symplectic form on $M$ with respect to $\F_{\g'}$. Let $v_1,\dots, v_k \in \g$ and suppose that there exists a smooth function $h_{v_i}$ such that $dh_{v_i} = -\iota_{X_{v_i}}\omega$ for $i=1,\dots, k$. Put $h = (h_1,\dots, h_k) \co M \to \R^k$. Then the followings hold:
		\begin{enumerate}
			\item[$(A_k)$] For $c \in \R^k$, the fiber $h^{-1}(c)$ is connected unless empty. 
			\item[$(B_k)$] $h(M)$ is convex. 
			\item[$(C_k)$] If $Z_1,\dots, Z_N$ are the connected components of the set of common critical points of $h_{v_i}$, then $h(Z_j)$ is a point $c_j$ and $h(M)$ is a convex hull of $c_1,\dots, c_N$. 
		\end{enumerate}
	\end{theo}
	\begin{proof}
		The proof consists of following steps.
		\begin{enumerate}
			\item[Step 1.] $(A_k)$ implies $(B_{k+1})$. 
			\item[Step 2.] $(A_k)$ holds by induction on $k$. 
			\item[Step 3.] $(B_k)$ implies $(C_k)$.
		\end{enumerate}
		
	 	Remark that it follows from the connectedness of $M$ that $(B_1)$ holds because $h(M)$ is a closed interval in $\R$. 
		
		Step 1. Assume that $(A_k)$ holds. Let $\pi \co \R^{k+1} \to \R^k$ be any linear projection given by $\pi(e_i) = \sum_{j=1}^ka_{ij}e_j$ for $i=1,\dots, k+1$. The composition $h' := \pi \circ h \co M \to \R^k$ satisfies the assumption of the theorem. Namely, $j$-th component of $h'$ is a smooth function $\sum_{i=1}^{k+1}a_{ij}h_{v_i}$, but 
		\begin{equation*}
			d\left(\sum_{i=1}^{k+1}a_{ij}h_{v_i}\right) = -\iota_{X_{\left(\sum_{i=1}^{k+1}a_{ij}v_i\right)}}\omega. 
		\end{equation*}
		Therefore each fiber of $h'$ is connected unless empty. Let $x, y \in h(M)$ and assume that $\pi$ is surjective and $\pi (x) = \pi (y) = c$. Since the fiber $\pi^{-1}(c)$ is a line in $\R^{k+1}$, it suffices to see that $h(M) \cap \pi^{-1}(c)$ is connected. Since $h' = \pi \circ h$, we have that $h(M) \cap \pi^{-1}(c) = h(h'^{-1}(c))$. Since $h$ is continuous and $h'^{-1}(c)$ is connected, $h(M) \cap \pi^{-1}(c)$ is connected, proving that $(A_k)$ implies $(B_{k+1})$. 
		
		Step 2. It follows from Lemmas \ref{lemm:Morse} and \ref{lemm:levelset} that $(A_1)$ holds. 
		Assume that $(A_k)$ holds. Let $v_1,\dots, v_{k+1} \in \mathfrak{g}$ and assume that there exists a smooth function $h_{v_i} \co M \to \R$ such that $dh_{v_i} = -\iota_{X_{v_i}}\omega$ for $i=1,\dots, k+1$. Let $h =(h_{v_1},\dots, h_{v_{k+1}})$ and let $c = (c_1,\dots, c_{k+1})$ be a point in $\R^{k+1}$. We want to show that $h^{-1}(c) = h_{v_1}^{-1}(c_1) \cap \dots \cap h_{v_{k+1}}^{-1}(c_{k+1})$ is connected unless empty. If $h$ has no regular value, then one of $dh_{v_i}$ is a linear combination of the others. By assumption that $(A_k)$ holds, we are done. Assume that $h$ has a regular value. Then, the set of regular values is dense in $h(M)$. By continuity, we only need to show that $h^{-1}(c)$ is connected for any regular value $c$. Then, $N := h_{v_1}^{-1} (c_1) \cap \dots \cap h_{v_k}^{-1}(c_k)$ is a connected submanifold by $(A_k)$. Moreover, it follows from Lemma \ref{lemm:reduction} that $\omega|_{N}$ is a transverse symplectic form on $N$ with respect to $\F_{\g''}$ on $N$, where $\g'' = \g + \R v_1+\dots + \R v_k$. The function $h_{v_{k+1}}|_N$ satisfies that $dh_{v_{k+1}}|_N = -\iota_{X_{v_{k+1}}}\omega|_N$. Therefore by Lemmas \ref{lemm:Morse} and \ref{lemm:levelset}, $(h_{v_{k+1}}|_N) ^{-1}(c_{k+1})$ is connected. Therefore $h^{-1}(c) = h_{v_1}^{-1}(c_1) \cap \dots \cap h_{v_{k+1}}^{-1}(c_{k+1})$ is connected, proving that $(A_k)$ holds for all $k$. 
		
		Step 3. The former assertion that states that $h(Z_j)$ is a point $c_j$ is obvious. Let $H$ be the closure of $\exp (\g'')$ in $G$. Let $x$ be a common critical point of $h_{v_1},\dots, h_{v_k}$. Since $(dh_{v_i})_x = (-\iota_{X_{v_i}}\omega)_x =0$, we have that $(X_{v_i})_x \in T_x\F_{\g'}$ for $i=1,\dots, k$. Therefore there exists $v_{i,x} \in \mathfrak{h}_x$ and $v_i' \in \g'$ such that $v_i = v_{i,x} + v_i'$ for $i=1,\dots, k$. Let $H_x^0$ denote the identity component of $H_x$. Then, $\{ \exp (t_1v_{1,x})\cdot \dots \cdot  \exp(t_kv_{k,x}) \mid t_i \in \R\}$ is dense in $H_x^0$ and $\exp (\mathfrak{h}_x +\g')$ is dense in $H$. Conversely, for a subtorus $H'$ of $H$, if $\exp (\mathfrak{h}' + \g')$ is dense in $H$, then each fixed point $x \in M^{H'}$ is a common critical point of $h_{v_1},\dots, h_{v_k}$. Let $u ' \in \g'$ and $v := \sum_{i=1}^ka_iv_i + u' \in \g''$ such that $\{\exp (tv) \mid t \in \R\}$ is dense in $H$. Put $h_v := \sum_{i=1}^ka_ih_{v_i}$. We claim that each critical point of $h_v$ is a common critical point of $h_{v_1},\dots, h_{v_k}$. Let $x$ be a critical point of $h_v$. Since $(dh_v)_x = (-\iota_{X_v}\omega)_x =0$, there exists $v_x \in \mathfrak{h}_x$ and $v' \in \g'$ such that $v = v_x + v'$. Since $\{\exp (tv)\mid t\in \R\}$ is dense in $H$, we have that $\{\exp (tv_x) \mid t \in \R\}$ is also dense in $H_x^0$. By definition of $v$ and $v_x$, the closure of $\exp (\mathfrak{h}_x + \g')$ is $H$. Therefore the critical point $x$ of $h_v$ is a common critical point of $h_{v_1},\dots, h_{v_k}$. In particular, $h_v$ takes the minimum value in a common critical point of $h_{v_1},\dots, h_{v_k}$. It turns out that the linear form $\alpha := \sum_{i=1}^k a_i e_i^*$ restricted to $h(M)$ takes the minimum value at one of $c_j$'s. Therefore
		\begin{equation}\Label{eq:hM}
			h(M) \subseteq \bigcap_{(a_1,\dots, a_k) \in A} \{ y = (y_1,\dots, y_k)\in \R^k \mid \langle \alpha, y\rangle \geq \min ( \langle \alpha, c_j\rangle \mid j=1,\dots, N ) \},  
		\end{equation}
		where 
		\begin{equation*}
			A := \{ (a_1,\dots, a_k) \mid \{\exp(tv) \mid t \in \R\} \text{ is dense in $H$} \}. 
		\end{equation*}
		Since $A$ is dense in $\R^k$, the right hand side of \eqref{eq:hM} is the convex hull of $c_j$'s. It follows from $(B_k)$ and $c_j \in h(M)$ for all $j$ that $h(M)$ is the convex full of $c_j$'s, proving the theorem. 
	\end{proof}

\section{Holomorphic foliations from torus actions}\Label{sec:holofoli}
	Let $M$ be a complex manifold and let $G$ be a compact torus acting on $M$ as holomorphic transformations. In this section, we define a subspace $\g_J$ of $\g$ which acts on $M$ local freely by using the complex structure $J$ on $M$ and the action of $G$. We begin with the following lemmas.
	\begin{lemm}\Label{lemm:eqholochart}
		Let $M$ be a  complex manifold equipped with an action of a compact torus $G$ which acts as holomorphic transformations. For $x \in M$, there exists $G_x$-invariant open neighborhoods $U$ at $x \in M$ and $V$ at $0 \in T_xM$ such that $U$ and $V$ are $G_x$-equivariantly biholomorphic. 
	\end{lemm} 
	\begin{proof}
		Let $U_0$ be an open neighborhood at $x$ and let $\varphi \co U_0 \to V_0$ be a local holomorphic coordinate centered at $x$, where $V_0$ is an open subset of $\C^n$. Since $G_x$ is compact, the intersection $\bigcap_{g \in G_x} g(U_0)$ is a $G_x$-invariant open neighborhood at $x$. By restricting the domain of definition, we may assume that $U_0$ is $G_x$-invariant. Through the differential $(d\varphi)_x \co T_xM \to T_0\C^n = \C^n$, we identify $\C^n$ with $T_xM$. Then we have a biholomorphism $(d\varphi)_x^{-1}\circ \varphi \co U_0 \to (d\varphi)_x^{-1}(V_0) \subseteq T_xM$. By averaging on $G_x$, we have a $G_x$-equivariant holomorphic map
		\begin{equation*}
			\varphi' := \int_{g \in G_x} (dg)_x \circ ((d\varphi)_x^{-1} \circ \varphi)\circ g^{-1} dg \co U_0 \to T_xM. 
		\end{equation*}
		$\varphi'$ is no longer injective, but, $(d\varphi')_x = \id_{T_xM}$. Therefore, it follows from the implicit function theorem that there exists an open subset $U$ of $U_0$ such that $\varphi'|_{U} \co U \to \varphi'(U)$ is biholomorphic. As before, we may assume that $U$ is $G_x$-invariant and then $V :=\varphi'(U)$ is also $G_x$-invariant. Therefore there exists a $G_x$-equivariant biholomorphism $\varphi' \co U \to V \subseteq T_xM$, proving the lemma. 
	\end{proof}
	\begin{lemm}\Label{lemm:ap}
		Let $M$ be a connected complex manifold with the complex structure $J$. Let $X$ be a nonzero almost periodic vector field on $M$ whose flows preserve $J$. If $X$ vanishes at a point $x \in M$, then $JX$ is not almost periodic. 
	\end{lemm}
	\begin{proof}
		Assume that $JX$ is almost periodic. Since $X$ is an infinitesimal automorphism of $J$, $[X, JX]=0$. Therefore we may assume that a compact torus $G$ acts on $M$ effectively and  as holomorphic transformations and there exist $v, v' \in \mathfrak{g}$ such that $X = X_v$, $JX = X_{v'}$ and the subgroup $\{ \exp (sv) \exp (tv') \mid s, t\in \R\}$ is dense in $G$. Since $X_x = (JX)_x =0$, $x$ is a $G$-fixed point. Let $\alpha_1,\dots, \alpha_n \in \Hom (G, S^1)$ be the weights of the $G$-representation $T_xM$. By Lemma \ref{lemm:eqholochart}, there exists an equivariant biholomorphic map $U \to V \subseteq T_xM$, where $U$ is an  open neighborhood at $x$. Combining with the decomposition of $T_xM$ into $1$-dimensional representations of weights $\alpha_1,\dots, \alpha_n$, we have a local coordinate $(z_1,\dots, z_n) \co U \to \C^n$ such that $z_i(g\cdot p) = \alpha_i(g)z_i(p)$ for $p \in U$. Let $x_i$ and $y_i$ denote the real and imaginary part of $z_i$, respectively. Then, through the local coordinate $(z_1,\dots, z_n)$ we can represent $X$ and $JX$ as 
		\begin{equation*}
			X = 2\pi \sum _{i=1}^n \langle d \alpha_i,v\rangle\left( -y_i\frac{\partial}{\partial x_i} + x_i\frac{\partial}{\partial y_i}\right) 
		\end{equation*}
		and 
		\begin{equation*}
			JX = 2\pi \sum _{i=1}^n \langle d \alpha_i,  v' \rangle\left( -y_i\frac{\partial}{\partial x_i} + x_i\frac{\partial}{\partial y_i}\right). 
		\end{equation*}
		On the other hand, $J$ is represented as 
		\begin{equation*}
			J = \sum_{i=1}^n \left( \frac{\partial}{\partial y_i} \otimes d x_i - \frac{\partial}{\partial x_i}\otimes d y_i\right). 
		\end{equation*}
		Therefore 
		\begin{equation*}
			\begin{split}
				0 &= X + J^2X\\
				 &= 2\pi \sum_{i=1}^n \left( \langle d \alpha_i, v\rangle \left( -y_i\frac{\partial}{\partial x_i} + x_i\frac{\partial}{\partial y_i}\right) + \langle d \alpha_i, v' \rangle \left(-x_i\frac{\partial}{\partial x_i}-y_i\frac{\partial}{\partial y_i} \right) \right)\\
				 &= 2\pi \sum_{i=1}^n \left(  \langle d\alpha_i, -y_iv-x_iv'\rangle \frac{\partial}{\partial x_i} +\langle d\alpha_i, x_iv-y_iv' \rangle  \frac{\partial}{\partial y_i}\right). 
			\end{split}
		\end{equation*}
		Therefore, by substituting $\epsilon$, $0<|\epsilon |<<1$ for $x_i$ and $y_i$, we have that 
		\begin{equation*}
				0 = \langle d\alpha_i, -\epsilon v-\epsilon v'\rangle  = -\epsilon \langle d\alpha_i, v + v'\rangle 
		\end{equation*}
		and 
		\begin{equation*}
			0 = \langle d\alpha_i, \epsilon v-\epsilon v'\rangle = \epsilon \langle d\alpha_i, v - v'\rangle 
		\end{equation*}
		for all $i=1,\dots, n$. Thus we have $\langle d\alpha_i, v\rangle = 0$ for all $i =1,\dots, n$. 
		Since the action of $G$ on $M$ is effective and $M$ is connected, $d \alpha_i \in \mathfrak{g}^*$ for $i=1,\dots, n$ spans $\mathfrak{g}^*$. This together with the fact that $\langle d\alpha_i,v\rangle = 0$ for all $i$ shows that $v = 0$. This contradicts the assumption that $X = X_v$ is nonzero and hence $JX$ is not almost periodic, as required. 
	\end{proof}
	\begin{prop}\Label{prop:localfree}
		Let $M$ be a connected complex manifold with the complex structure $J$. Let $G$ be a compact torus acting on $M$ effectively and as holomorphic transformations. Define 
		\begin{equation*}
			\mathfrak{g}_J := \{ v \in \mathfrak{g} \mid \text{there exists $v' \in \mathfrak{g}$ such that $X_v = -JX_{v'}$}\}.
		\end{equation*}
		Then, 
		\begin{enumerate}
			\item $\mathfrak{g}_J$ is a Lie subalgebra of $\mathfrak{g}$. 
			\item $\mathfrak{g}_J$ has the complex structure $J_0$ which satisfies $X_{J_0(v)} = JX_{v}$. 
			\item $\mathfrak{g}_J$ acts on $M$ holomorphically and local freely. 
		\end{enumerate}
	\end{prop}
	\begin{proof}
		Part (1) follows from the fact that $G$ is commutative. For Part (2), let $v \in \mathfrak{g}_J$. Assume that $v', v'' \in \mathfrak{g}_J$ satisfy that $X_v = -JX_{v'} = -JX_{v''}$. Then, $X_{v'} = X_{v''}$. It follows from the effectiveness of the $G$-action that $v' = v''$. Therefore for $v \in \mathfrak{g}_J$, there exists unique $J_0(v) \in \mathfrak{g}_J$ such that $X_{J_0(v)} = JX_{v}$. The map $J_0 \co \mathfrak{g}_J \to \mathfrak{g}_J$ is linear and $J_0^2 = -1$, proving Part (2). Part (3) follows from Part (2) and Lemma \ref{lemm:ap}. The proposition is proved. 
	\end{proof}
	Let $M$ be a connected complex manifold with the complex structure $J$ and let a compact torus $G$ act on $M$ effectively and as holomorphic transformations. By Proposition \ref{prop:localfree}, $\mathfrak{g}_J$ acts on $M$ holomorphically and local freely. Therefore we have a holomorphic foliation $\F_{\mathfrak{g}_J}$ whose leaves are $\mathfrak{g}_J$-orbits. 
\section{Torus invariant transverse K\"{a}hler foliations}\Label{sec:invtransverse}
	
	A transverse K\"{a}hler form is a special kind of transverse symplectic form. Let $M$ be a complex manifold with the complex structure $J$.  Let $\F$ be a holomorphic foliation on $M$. A real $2$-form $\omega$ on $M$ is called \emph{transverse K\"{a}hler} with respect to $\F$ if the following conditions are satisfied: 
	\begin{enumerate}
		\item $\omega$ is transverse symplectic with respect to $\F$. 
		\item $\omega$ is of type $(1,1)$. Namely, For $Y_x,Z_x \in T_xM$, $\omega_x(JY_x,JZ_x) = \omega_x(Y_x,Z_x)$.
		\item $\omega$ is positive. Namely, $\omega_x(Y_x, JY_x) \geq 0$ for all $Y_x \in T_xM$. 
	\end{enumerate}
	The conditions (1) and (3) imply that $\omega_x(Y_x, JY_x)= 0$ if and only if $Y_x \in T_x\F$. 
	For a holomorphic foliation $\F$ on $M$, if a transverse K\"{a}hler form $\omega$ exists, we say that $\F$ is \emph{transverse K\"{a}hler}. 
	\begin{prop}\Label{prop:average}
		Let $M$ be a complex manifold with the complex structure $J$. Let $G$ be a compact torus acting on $M$ as holomorphic transformations. Let $\F$ be a $G$-invariant foliation and let $\omega$ be a transverse K\"{a}hler form with respect to $\F$. Then, 
		\begin{equation*}
			\int _{g \in G} g^*\omega dg
		\end{equation*}
		is a transverse K\"{a}hler form with respect to $\F$ and invariant under the $G$-action on $M$. 
	\end{prop}
	\begin{proof}
		For short, denote 
		\begin{equation*}
			\omega' = \int _{g \in G} g^*\omega dg. 
		\end{equation*}
		Since $\omega$ is closed, so is $\omega'$. Since $G$ acts on $M$ preserving the complex structure $J$, $\omega'$ is a positive $(1,1)$-form. It remains to show that $\ker \omega'_x = T_x\F$ for all $x \in M$. By definition, for $Y_x \in T_xM$, 
		\begin{equation*}
			\begin{split}
				\omega'_x (Y_x, JY_x) &= \int _{g \in G} (g^*\omega)_x (Y_x, JY_x) dg\\
				&= \int _{g \in G} \omega_{g\cdot x} ((dg)_x(Y_x) , (d g)_x (JY_x))d g \\
				&= \int _{g \in G} \omega_{g\cdot x} ((d g)_x(Y_x) , J((d g)_x (Y_x)))d g
			\end{split}
		\end{equation*}
		because $J$ is $G$-invariant. Since $\omega_{g\cdot x} ((d g)_x(Y_x) , J((d g)_x (Y_x))) \geq 0$ and the equality holds if and only if $(d g)_x(Y_x) \in T_{g\cdot x}\F$, it follows from the $G$-invariance of $\F$ that $\omega'_x (Y_x,JY_x) = 0$ if and only if $Y_x \in T_x\F$, proving the proposition. 
	\end{proof}
	Thanks to Proposition \ref{prop:average}, if a $G$-invariant foliation $\F$ is transverse K\"{a}hler, we may always assume that the transverse K\"{a}hler form with respect to $\F$ is $G$-invariant without loss of generality. 

	For foliations $\F_1$ and $\F_2$ on a smooth manifold $M$, we denote by $\F_1 \subseteq \F_2$ if $T\F_1 \subseteq T\F_2$. Our next purpose is to  give a lower bound of $G$-invariant transverse K\"{a}hler foliations that admit moment maps. 
	\begin{prop}\Label{prop:lowerbound}
		Let $M$ be a connected complex manifold with the complex structure $J$. Let a compact torus $G$ act on $M$ effectively and as holomorphic transformations. Let $\F$ be a $G$-invariant holomophic foliation and let $\omega$ be a $G$-invariant transverse K\"{a}hler form with respect to $\F$. If there exists a moment map with respect to $\omega$, then $\F_{\mathfrak{g}_J} \subseteq \F$. 
	\end{prop}
	\begin{proof}
		Let $\Phi \co M \to \g^*$ be a moment map. We denote by $h_v$ the smooth function given by $h_v(x) = \langle \Phi (x), v\rangle$ for $v \in \g$ and $x \in M$. Since $h_v$ is $G$-invariant for $v \in \g$, we have that 
		\begin{equation}\Label{eq:isotropic}
			0 = \mathcal{L}_{X_{v_1}}h_{v_2} = \iota_{X_{v_1}}dh_{v_2} = -\iota_{X_{v_1}}\iota_{X_{v_2}}\omega = 2\omega(X_{v_1}, X_{v_2})
		\end{equation}
		 for any $v_1,v_2 \in \g$. Assume that $v \in \g_J$. Then, 
		 \begin{equation*}
		 	0 \leq \omega (X_v, JX_v) = \omega (X_v, X_{J_0(v)}) = 0
		 \end{equation*}
		 by \eqref{eq:isotropic}. Therefore $(X_v)_x \in T_x\F$ for all $x$ and hence $\F_{\g_J} \subseteq \F$, as required. 
	\end{proof}
	\begin{theo}\Label{theo:liftedmoment}
		Let $M$, $J$, $G$, $\F$, $\omega$ be as Proposition \ref{prop:lowerbound}. Let $q \co \mathfrak{g} \to \mathfrak{g}/\mathfrak{g}_J$ be the quotient map. Assume that there exists a moment map $\Phi \co M \to \mathfrak{g}^*$ with respect to $\omega$. Then, there exist $c \in \mathfrak{g}^*$ and a smooth map $\widetilde{\Phi} \co M \to (\mathfrak{g}/\mathfrak{g}_J)^*$ such that $\Phi + c = q^* \circ \widetilde{\Phi}$. 
	\end{theo}
	\begin{proof}
		For $v \in \g$, $h_v$ denotes the smooth function given by $h_v(x) = \langle \Phi(x), v\rangle$. It follows from Proposition \ref{prop:lowerbound} that $dh_v =-\iota_{X_v}\omega = 0$ for $v \in \g_J$. It turns out that $h_v$ is constant on $M$. Let $i \co \g_J \to \g$ denote the inclusion. Then, there exists $\overline{c} \in \mathfrak{g}_J^*$ such that $i^*\circ \Phi (x)= \overline{c}$ for any $x \in M$. The sequences
		\begin{equation*}
			\xymatrix{
				0 \ar[r] & \mathfrak{g}_J \ar[r]^i & \mathfrak{g} \ar[r]^q & \mathfrak{g}/\mathfrak{g}_J \ar[r] & 0
			}
		\end{equation*}
		and 
		\begin{equation*}
			\xymatrix{
				0  & \mathfrak{g}_J^* \ar[l] & \mathfrak{g}^*  \ar[l]^{i^*} & (\mathfrak{g}/\mathfrak{g}_J)^*  \ar[l]^{q^*}& 0 \ar[l]
			}
		\end{equation*}
		are exact. Since $i^*$ is surjective, there exists $c \in \mathfrak{g}^*$ such that $i^* (c) = \overline{c}$. In particular, $i^* (\Phi (x) - c) = 0$ for all $x \in M$. Therefore, there uniquely exists $\widetilde{\Phi}(x) \in (\mathfrak{g}/\mathfrak{g}_J)^*$ such that $q^*(\widetilde{\Phi} (x)) = \Phi (x)-c$ for all $x \in M$. The smoothness is obvious. The theorem is proved. 
	\end{proof}
	We call $\widetilde{\Phi} \co M \to (\g/\g_J)^*$ a \emph{lifted moment map}. As a corollary of Theorems \ref{theo:convexity} and \ref{theo:liftedmoment}, we have the following. 
	\begin{coro}\Label{coro:liftedconvexity}
		Let $M$ be a compact connected complex manifold. Let a compact torus $G$ act on $M$ effectively and preserving the complex structure $J$ on $M$. Assume that $\F_{\g_J}$ is transverse K\"{a}hler and there exists a moment map $\Phi \co M \to \g^*$ with respect to a $G$-invariant transverse K\"{a}hler form. Then, the image of $M$ by a lifted moment map $\widetilde{\Phi} \co M \to (\g/\g_J)^*$ is a convex polytope in $(\g/\g_J)^*$. 
	\end{coro}

\section{The extreme case}\Label{sec:extreme}
	In this section, we consider the extreme case. First we recall the notion of maximal torus action introduced in \cite{Ishida}. Let $M$ be a connected smooth manifold equipped with an effective action of a compact torus $G$. Then, for any point $x$, we have that $\dim G_x + \dim G \leq \dim M$. The $G$-action on $M$ is \emph{maximal} if there exists a point $x \in M$ such that 
	\begin{equation*}
		\dim G + \dim G_x = \dim M. 
	\end{equation*}
	Any compact connected complex manifold $M$ equipped with a maximal action of a compact torus $G$ which preserves the complex structure can be described with a fan $\Delta$ in $\g$ and a complex subspace $\mathfrak{h}$ of $\g^\C$. 
	\begin{theo}[see \cite{Ishida}]\Label{theo:maximal}
		Let $M$ be a compact connected complex manifold $M$ equipped with a maximal action of a compact torus $G$ which preserves the complex structure $J$. Then, there exists a nonsingular fan $\Delta$ in $\g$ and a complex subspace $\mathfrak{h}$ such that $M$ is $G$-equivariantly biholomorphic to $X(\Delta)/H$, where $X(\Delta)$ denotes the toric variety associated with $\Delta$ and $H := \exp (\mathfrak{h}) \subseteq G^\C \curvearrowright X(\Delta)$. 
	\end{theo}
	We shall recall how to deduce $\Delta$ and $\mathfrak{h}$ from $M$ briefly. Each connected component of the set of fixed points of a circle subgroup of $G$ is a closed complex submanifold of $M$. If such a submanifold has complex codimension one, then we call it a \emph{characteristic submanifold} of $M$. The number of characteristic submanifolds is at most finite. Let $N_1,\dots, N_k$ be characteristic submanifolds of $M$. Each characteristic submanifold $N_i$ is fixed by a circle subgroup $G_i$ of $G$ by definition. To each characteristic submanifold $N_i$, we assign a group isomorphism $\lambda_i \co S^1 \to G_i \subseteq G$ such that 
	\begin{equation*}
		(\lambda_i(g))_*(\xi) =g\xi \quad \text{for all $g \in S^1$ and $\xi \in TM|_{N_i}/TN_i$}. 
	\end{equation*}
	We can think of $\lambda \in \Hom (S^1,G)$ as a vector in $\g$ by $d\lambda(1) \in \g$. We have a collection $\Delta$ of cones 
	\begin{equation*}
		\Delta := \left\{ \pos (\lambda_i \mid i \in I) \mid \bigcap_{i \in I} N_i \neq \emptyset\right\},
	\end{equation*}
	where $\pos (\lambda_i \mid i \in I)$ is the cone spanned by $\lambda_i$ for $i \in I$. It has been shown that $\Delta$ is a nonsingular fan in $\g$ with respect to the lattice $\Hom (S^1,G)$. Since the action of $G$ preserves the complex structure $J$ on $M$, it extends to a holomorphic action of $G^\C$ on $M$. Then the complex subspace $\mathfrak{h}$ of $\g^\C = \g \otimes \C = \g \otimes 1 + \g \otimes \sqrt{-1}$ is defined to be the Lie algebra of global stabilizers of the $G^\C$-action on $M$. Namely, 
	\begin{equation*}
		\mathfrak{h} = \{ u\otimes 1 +v \otimes \sqrt{-1} \in \g^\C \mid X_u +JX_v = 0\}. 
	\end{equation*}
	The pair of $\Delta$ and $\mathfrak{h}$ satisfies the followings. 
	\begin{enumerate}
		\item The restriction $p|_\mathfrak{h}$ of the projection $p \co \g^\C \to \g\otimes 1 \cong \g$ is injective. 
		\item The quotient map $q \co \g \to \g/p(\mathfrak{h})$ sends $\Delta$ to a complete fan $q(\Delta)$ in $\g/p(\mathfrak{h})$. 
	\end{enumerate}
	Conversely, if $\Delta$ and $\mathfrak{h}$ satisfy the conditions (1) and (2), then the quotient $X(\Delta)/H$ is a compact connected complex manifold and the action of $G$ on $X(\Delta)$ descends to a maximal action on $X(\Delta)/H$. 
	\begin{prop}\Label{prop:phgJ}
		Let $M$ be a compact connected complex manifold $M$ equipped with an action of a compact torus $G$ which preserves the complex structure $J$. Let $\mathfrak{h}$ be the Lie algebra of global stabilizers of the $G^\C$-action on $M$. Then, $p(\mathfrak{h}) = \g_J$. 
	\end{prop}
	\begin{proof}
		This follows from the definitions of $\mathfrak{h}$ and $\g_J$ immediately. 
	\end{proof}
	\begin{lemm}\Label{lemm:normal}
		Let $\Delta,\mathfrak{h}, q$ be as above and let $J$ denote the complex structure on $X(\Delta)/H$. Assume that $\F_{\g_J}$ is a transverse K\"{a}hler foliation on $X(\Delta)/H$ and let $\omega$ be a $G$-invariant transverse K\"{a}hler form with respect to $\F_{\g_J}$. In addition, assume that there exists a moment map $\Phi \co X(\Delta)/H \to \g^*$ with respect to $\omega$. Then, the image of $X(\Delta)/H$ by a lifted moment map $\widetilde{\Phi}$ is a convex polytope and $\widetilde{\Phi} (X(\Delta)/H)$ is a normal polytope of $q(\Delta)$. 
	\end{lemm}
	Before the proof of Lemma \ref{lemm:normal}, we shall recall notions of normal fan and normal polytope. Let $P$ be an $n$-dimensional polytope in a vector space $V^*$ of dimension $n$. For a vector $\alpha \in V$, we put
	\begin{equation*}
		F_v := \{ \alpha \in P \mid \langle \alpha, v\rangle \leq \langle \alpha', v\rangle \text{ for all $\alpha' \in P$}\}. 
	\end{equation*}
	$F_v$ is a face of $P$ that attains the minimum value of $v$. For a face $F$ of $P$, the \emph{(inner) normal cone} $\sigma_F$ of $F$ is given by
	\begin{equation*}
		\sigma_F := \{ v \in V \mid F_v \subseteq F\}. 
	\end{equation*}
	Its relative interior is given by $\{ v \in V \mid F_v = F\}$. The \emph{(inner) normal fan} of $P$ is the correction $\Delta_P := \{\sigma_F\}_F$ of cones $\sigma_F$ for the faces $F$ of $P$. Conversely, for given fan $\Delta$ in $V$, a polytope $P$ in $V^*$ whose (inner) normal fan coincides with $\Delta$ is called an \emph{(inner) normal polytope} of $\Delta$. If such a polytope $P$ exists for $\Delta$, then $\Delta$ is said to be \emph{polytopal}. 

	Now assume that $\Delta$ is a nonsingular fan in $\g$. 
	We also prepare several notations of submanifolds. For a cone $\sigma \in \Delta$, we denote by $X_\sigma$ the closed toric subvariety of $X(\Delta)$ corresponds to $\sigma$. If $\lambda_1,\dots, \lambda_k \in \Hom (S^1,G)$ be the primitive generators of $1$-cones of $\Delta$, $\sigma$ can be written as $\sigma = \pos (\lambda_i \mid i \in I)$ for some $I \subseteq \{1,\dots, k\}$ and $(\lambda_i)_{i \in I}$ is a part of $\Z$-basis of $\Hom (S^1,G)$. More precisely, each point of $X_\sigma$ is fixed by a subtorus $G_\sigma$ of $G$, and $(\lambda_i)_{i \in I}$ is a $\Z$-basis of $\Hom (S^1, G_\sigma)$. The image $Y_\sigma$ of $X_\sigma$ by the quotient map $X(\Delta)\to X(\Delta)/H$ is a closed submanifold of $X(\Delta)/H$ because $X_\sigma$ and $Y_\sigma$ both are connected components of the set of fixed points by the $G_\sigma$-actions. If we denote by $(\alpha^I_i)_{i \in I}$ the dual basis of $(\lambda_i)_{i \in I}$, the set of nonzero weights of $T_xY_\sigma$ coincides with $(\alpha^I_i)_{i \in I}$ for all $x \in Y_\sigma$.  
	\begin{proof}[Proof of Lemma \ref{lemm:normal}]
		We may assume that $q^*\circ \widetilde{\Phi} = \Phi$ without loss of generality. For $v \in \g$, define $h_v \co X(\Delta)/H \to \g^*$ by $h_v(x) = \langle \Phi (x), v \rangle$. Then, $h_v (x) = \langle \widetilde{\Phi}(x),q(v)\rangle$. We shall see that each connected component of the set of critical points of $h_v$ is one of $Y_\sigma$ for some $\sigma \in \Delta$. Let $x \in X(\Delta)/H$. Since $dh_v = -\iota_{X_v}\omega$, $x$ is a critical point of $h_v$ if and only if$(X_v)_x \in T_x\F_{\g_J}$, in particular, $v \in \g_x + \g_J$. Therefore, the set of critical points is 
		\begin{equation*}
			\bigcup _{\sigma ; v \in \g_\sigma +\g_J}Y_\sigma. 
		\end{equation*}
		Assume that $h_v$ takes the minimum value $a_v$ on $Y_\sigma$. Let $v_\sigma \in \g_\sigma$ such that $q(v) = q(v_\sigma)$. Then, $\langle d\alpha_i^I, v_\sigma\rangle > 0$ for all $i \in I$, where $(\lambda_i)_{i \in I}$ is the set of primitive generators of $\sigma$ (see Remark \ref{rema:localminimum}). Therefore $v _\sigma$ sits in the relative interior of $\sigma$. In particular, $q(v)$ sits in the relative interior of $q(\sigma)$. The converse is also true; if $q(v)$ sits in the relative interior of $q(\sigma)$, then $h_v$ takes the minimum value $a_v$ on $Y_\sigma$. 
		
		By Corollary \ref{coro:liftedconvexity}, $\widetilde{\Phi} (X(\Delta)/H)$ is a convex polytope $P$ in $(\g/\g_J)^*$. We claim that, for each $\sigma$, the image of $Y_\sigma$ by $\widetilde{\Phi}$ is a face of $P$. Let $v \in \g$ such that $q(v)$ sits in the relative interior of $\sigma$. Then, $h_v$ takes the minimum value $a_v$ on $Y_\sigma$. But $h_v(x) = \langle \widetilde{\Phi}(x), q(v)\rangle$ implies that $x$ attains the minimum value $a_v$ of $h_v$ if and only if $\widetilde{\Phi}(x)$ attains the minimum value $a_v$ of $q(v)|_{P}$. Since $h_v^{-1}(a_v) = Y_\sigma$, we have that $(q(v)|_P)^{-1}(a_v) = \widetilde{\Phi}(Y_\sigma)$. Since $(q(v)|_P)( \alpha) \geq a_v$ for all $\alpha \in P$, $\widetilde{\Phi}(Y_\sigma)$ is a face of $P$ that is given by $P \cap H_{q(v), a_v}$, where $H_{q(v), a_v}$ is the hyperplane in $(\g/\g_J)^*$ defined as $H_{q(v), a_v} := \{ \alpha \in (\g/\g_J)^* \mid \langle \alpha, q(v)\rangle = a_v \}$. 
		
		Conversely, if a face $F$ of $P$ is given by $P \cap H_{q(v), a_v}$, then $F$ is the image of $Y_\sigma$ by $\widetilde{\Phi}$, where $\sigma$ is the cone such that $q(v)$ sits in the relative interior of $q(\sigma)$. It turns out that for each face $F$ of $P$ there exists a cone $\sigma$ such that the inner normal cone of $F$ coincides with $q(\sigma)$. Hence $P$ is a normal polytope of $q(\Delta)$, as required. 
	\end{proof}
	Now we consider the obstruction for the existence of a moment map in case of LVMB manifold with indispensable integer $0$. Let $\Sigma$ be an abstact simplicial complex on $\{0, 1,\dots, m\}$ (a singleton $\{i\}$ does not need to be a member of $\Sigma$). Let $G = (S^1)^m$. Then $\g = \R^m$ and $\Z^m$ is identified with $\Hom (S^1,G)$. $G$ acts on $\C P^m$ via $(g_1,\dots, g_m) \cdot [z_0, z_1,\dots, z_m] := [z_0, g_1z_1,\dots, g_mz_m]$ for $(g_1,\dots, g_m) \in G$ and $[z_0, z_1,\dots, z_m] \in \C P^m$. Put $e_0 := -e_1-\dots -e_m$ and define 
	\begin{equation*}
		\Delta := \{ \pos (e_i \mid i \in I) \mid I \in \Sigma\}. 
	\end{equation*}
	$\Delta$ is a nonsingular fan in $\R^m$ and the toric variety $X(\Delta)$ associated with $\Delta_\Sigma$ is given by 
	\begin{equation*}
		X(\Delta_\Sigma) = \bigcup_{I \in \Sigma}U_I, 
	\end{equation*}
	where 
	\begin{equation*}
		U_I = \{ [z] = [z_0,\dots, z_m] \in \C P^m \mid z_j \neq 0 \text{ if $j \notin I$}\}. 
	\end{equation*}
	Let $\mathfrak{h} \subseteq \C^m$ such that $\Delta :=\Delta_\Sigma$ and $\mathfrak{h}$ satisfy the conditions (1) and (2). We call the manifold $X(\Delta_\Sigma)/H$ an \emph{LVMB manifold}. Moreover, if $q(\Delta_\Sigma)$ is polytopal, we call it an \emph{LVM manifold}, due to \cite[Theorems 2.2 and 3.10]{Battisti}. 
	
	If an integer $i$ satisfies that $\{i\} \notin \Sigma$, we say that $i$ is \emph{indispensable}, according to the literature of LVMB manifolds (see \cite{Bosio}, \cite{Meersseman} and \cite{Meersseman-Verjovsky}). In case when $0$ is indispensable, $[z] \in X(\Delta_\Sigma)$ implies that $z_0 \neq 0$. Therefore we can think of $X(\Delta_\Sigma)$ as an open subset of $\C^m$ via the map 
	\begin{equation*}
		[z_0,\dots, z_m] \mapsto \left(\frac{z_1}{z_0},\dots, \frac{z_m}{z_0}\right). 
	\end{equation*}
	Namely, 
	\begin{equation*}
		X(\Delta_\Sigma) = \bigcup_{I \in \Sigma}U_I',
	\end{equation*}
	where 
	\begin{equation*}
		U_I' = \{ z = (z_1,\dots, z_m) \in \C^m \mid z_j \neq 0 \text{ if } j \notin I\}. 
	\end{equation*}
	In case when $0$ is indispensable, we call $X(\Delta_\Sigma)/H$ an \emph{LVMB manifold with indispensable integer $0$}.  
	
	It has been shown in \cite{Battaglia-Zaffran} that the odd degree basic cohomology groups of $X(\Delta_\Sigma)/H$ with respect to $\F_{\g_J}$ vanish for shellable $\Sigma$. Therefore for shellable $\Sigma$, there exists a moment map for any transverse K\"{a}hler form on $X(\Delta_\Sigma)/H$ with respect to $\F_{\g_J}$. We can avoid the assumption on $\Sigma$ for the vanishing of first basic cohomology groups with a straightforward computation.  
	\begin{lemm}\Label{lemm:basic}
		Let $X(\Delta_\Sigma)/H$ be an LVMB manifold with indispensable integer $0$. Let $J$ be the complex structure on $X(\Delta_\Sigma)/H$. Then, $H_{\F_{\g_J}}^1(X(\Delta_\Sigma)/H) =0$. 
	\end{lemm}
	\begin{proof}
		Assume that, $\{i\}$ is a member of $\Sigma$ for $i=1,\dots, r$ but not for $i = r+1,\dots, m$. Then, $X(\Delta_\Sigma) = X(\Delta_{\Sigma'}) \times (\C \setminus \{0\})^{m-r}$, where $\Sigma'$ is an abstract simplicial complex on $\{0,\dots, r\}$ such that if $I \in \Sigma$ then $I \in \Sigma'$. $X(\Delta_{\Sigma'})$ is a complement of coordinate subspaces of real codimension $\geq 4$ in $\C^r$. Thus, $X(\Delta_{\Sigma'})$ is simply connected. Let $c_i \co S^1 \to X(\Delta_\Sigma)$ be the curve defined by 
		\begin{equation*}
			c_i (t) = (\underbrace{1,\dots, 1}_{i-1}, t, \underbrace{1,\dots, 1}_{m-i}) \in X(\Delta_\Sigma) \subseteq \C^m
		\end{equation*}
		for $i=1,\dots, m$. $c_i$ is null-homologous for $i=1,\dots, r$ and the homology classes $[c_i]$ determined by $c_i$ for $i=r+1,\dots, m$ form a basis of $H_1(X(\Delta_\Sigma))$. 
	
		Let $\beta$ be a $1$-form on $X(\Delta_\Sigma)/H$. 
		$\beta$ is closed and basic for $\F_{\g_J}$ if and only if $d\beta =0$ and $\iota_{X_v}\beta =0$ for any $v \in \g_J$. Let $\pi \co X(\Delta_\Sigma) \to X(\Delta_\Sigma)/H$ be the quotient map. $\pi^*\beta$ is a $1$-form basic for $\F_{\mathfrak{h}}$. That is, $\iota_{X_u}\pi^*\beta = 0$ and $\iota_{X_u}d\pi^*\beta =0$ for $u \in \mathfrak{h}$. Therefore we need to show that a closed $1$-form $\gamma$ on $X(\Delta_\Sigma)$ satisfying
		\begin{itemize}
			\item $\iota_{X_u}\gamma = 0$ for $u \in \mathfrak{h}$, 
			\item $\iota_{X_v}\gamma = 0$ for $v \in \g_J$. 
		\end{itemize}
		is exact. Let $\gamma$ be such a $1$-form on $X(\Delta_\Sigma)$. By averaging $\gamma$ with the action of $G$, we may assume that $\gamma$ is $G$-invariant without loss of generality. Since $\gamma$ is real, we can represent 
		\begin{equation*}
			\gamma = \sum_{i=1}^m f_idz_i +\overline{f_i}d\overline{z_i}
		\end{equation*}
		with smooth functions $f_i \co X(\Delta_\Sigma) \to \C$. 
		Let $v = (v_1,\dots, v_m) \in  \g =\R^m$. Then $X_v$ can be represented as 
		\begin{equation}\Label{eq:Xv}
			X_v = 2\pi\sum_{i=1}^m \sqrt{-1}v_i\left(z_i\frac{\partial}{\partial z_i} - \overline{z_i}\frac{\partial}{\partial \overline{z_i}}\right). 
		\end{equation}
		Since $\gamma$ is $G$-invariant, we have that there exists $a_i \in \R$ such that $2\pi\sqrt{-1}(z_if_i -\overline{z_i}\overline{f_i}) =a_i$ for $i=1,\dots, m$. Since $H_1(X(\Delta_\Sigma))$ is generated by $c_i$ for $i=r+1,\dots, m$ and the Kronecker pairing is given by $\langle [c_i], [\gamma] \rangle =a_i$, it suffices to show that $a_i = 0$ for $i = r+1,\dots, m$.

		If $u = (u_{1,1} + \sqrt{-1}u_{1,\sqrt{-1}}, \dots, u_{m,1}+\sqrt{-1}u_{m,\sqrt{-1}}) \in \C^m$, $X_u$ can be represented as 
		\begin{equation}\Label{eq:Xu}
			X_u = 2\pi\sum_{i=1}^m\left( \sqrt{-1}u_{i,1}\left(z_i\frac{\partial}{\partial z_i} - \overline{z_i}\frac{\partial}{\partial \overline{z_i}}\right) - u_{i,\sqrt{-1}}\left(z_i\frac{\partial}{\partial z_i} + \overline{z_i}\frac{\partial}{\partial \overline{z_i}}\right)\right).
		\end{equation}
		Since $p(\mathfrak{h}) =\g_J$ by Proposition \ref{prop:phgJ}, it follows from \eqref{eq:Xv} and \eqref{eq:Xu} that the conditions $\iota_{X_v}\gamma =0$ for $v \in \g_J$ and $\iota_{X_u}\gamma = 0$ for $u \in \mathfrak{h}$ are equivalent to 
		\begin{equation}\Label{eq:vzf}
			\sum_{i=1}^m v_iz_if_i =0 \quad \text{for all $v = (v_1,\dots, v_m) \in \g_J$}. 
		\end{equation}
		Assume that $\{1,\dots, n\} \in \Sigma$ is a maximal simplex. Then, $q(e_1), \dots, q(e_n)$ form a basis of $\g/\g_J$. Let $\alpha_1,\dots, \alpha_n$ be the dual basis of $q(e_1),\dots, q(e_n)$. Then, we have a basis 
		\begin{equation*}
			e_j - \sum _{i=1}^n \langle \alpha_i, q(e_j)\rangle e_i \quad \text{for $j=n+1,\dots, m$}
		\end{equation*}
		of $\g_J = \ker q$. Therefore we have that \eqref{eq:vzf} is equivalent  to 
		\begin{equation}\Label{eq:zf}
			z_jf_j -\sum_{i=1}^n \langle \alpha_i, q(e_j)\rangle z_if_i = 0 \quad \text{for $j=n+1,\dots, m$}. 
		\end{equation}
		The Kronecker pairing $\langle [c_i], [\gamma]\rangle = 0$ for $i =1,\dots, r$ because $c_i$ for $i=1,\dots,r$ is null-homologous. Therefore $a_i = 2\pi\sqrt{-1}(z_if_i-\overline{z_i}\overline{f_i}) = 0$ for $i =1,\dots, n$. This together with \eqref{eq:zf} yields that $a_j = 0$ for all $j=n+1,\dots, m$. Therefore $\gamma$ is exact, proving the lemma. 
	\end{proof}
	\begin{coro}\Label{coro:necessary}
		Assume that $\F_{\g_J}$ on $X(\Delta_\Sigma)/H$ is transverse K\"{a}hler. Then, the complete fan $q(\Delta)$ in $\g/\g_J$ is polytopal. Namely, $X(\Delta_\Sigma)/H$ is an LVM manifold. 
	\end{coro}
	\begin{proof}
		Let $\omega$ be a transverse K\"{a}hler form on $X(\Delta_\Sigma)/H$ with respect to $\F_{\g_J}$. Since $\F_{\g_J}$ is $G$-invariant, we may assume that $\omega$ is $G$-invariant by Proposition \ref{prop:average}. The closed $1$-form $-\iota_{X_v}\omega$ is exact for all $v \in \g$ by Lemma \ref{lemm:basic}. Therefore there exists a moment map on $X(\Delta_\Sigma)/H$ with respect to $\omega$. Let $\widetilde{\Phi} \co X(\Delta_\Sigma)/H \to (\g/\g_J)^*$ be a lifted moment map. By Lemma \ref{lemm:normal}, the image of $X(\Delta_\Sigma)/H$ by $\widetilde{\Phi}$ is a normal polytope of $q(\Delta)$. Therefore $q(\Delta)$ is polytopal, as required. 
	\end{proof}
	Conversely, we can construct a transverse K\"{a}hler form on $X(\Delta_\Sigma)/H$ with respect to $\F_{\g_J}$ from a normal polytope $P$ of $q(\Delta)$. Essentially, this fact has been shown in \cite{Loeb-Nicolau} and \cite{Meersseman}. But, the ``language" in this paper is slightly different from them. For reader's convenience, we give a brief explanation of the construction of a transverse K\"{a}hler form without a proof. Let $P$ be a normal polytope of $q(\Delta_\Sigma)$ represented as 
	\begin{equation*}
		P = \{ \alpha \in (\g/\g_J)^* \mid \langle \alpha, q(e_i)\rangle \geq a_i\}. 
	\end{equation*}
	The map $q^* \co (\g/\g_J)^* \to \g^*$ is an injective map. We consider the embedding $P \to \g^*$ given by
	\begin{equation*}
		\alpha \mapsto \sum_{i=1}^m(\langle \alpha, q(e_i)\rangle -a_i)e_i^* = q^*(\alpha) -\sum_{i=1}^ma_ie_i^*,  
	\end{equation*}
	where $e_i^*$ denotes the $i$-th dual basis vector of the standard basis $e_1,\dots, e_m$ of $\g = \R^m$. Let $i \co \g_J \to \g$ be the inclusion and consider the dual map $i^* \co \g^* \to \g_J^*$. The image of embedded $P$ is the point $i^*(\sum_{i=1}^m-a_ie_i^*) =: \beta$.
	$X(\Delta_\Sigma)$ is an open subset of $\C^m$. So $X(\Delta_\Sigma)$ has the standard K\"{a}hler form 
	\begin{equation*}
		\omega_{\text{st}} = \frac{\sqrt{-1}}{2}\sum_{i=1}^m dz_i \wedge d\overline{z_i}.
	\end{equation*}
	$G$ acts on $X(\Delta_\Sigma)$ preserving $\omega_{\text{st}}$. The map $\Phi \co X(\Delta_\Sigma) \to \g^*$ given by 
	\begin{equation*}
		\Phi (z_1,\dots, z_m) = \pi \sum_{i=1}^m|z_i|^2e_i^*
	\end{equation*}
	is a moment map with respect to $\omega_{\text{st}}$. For the compostion $i^* \circ \Phi \co X(\Delta_\Sigma) \to \g_J^*$, the value $\beta \in \g_J^*$ is a regular value and $(i^*\circ \Phi)^{-1}(\beta) =: \mathcal{Z}_P$ is a smooth manifold equipped with an action of $G$ and the $G$-invariant transverse symplectic form $\omega := \omega_{\text{st}}|_{\mathcal{Z}_P}$ with respect to $\F_{\g_J}$. Each orbit of $H$ intersects with $\mathcal{Z}_P$ at exactly one point in $\mathcal{Z}_P$, and hence the inclusion $\mathcal{Z}_P \to X(\Delta_\Sigma)$ induces an equivariant diffeomorphism $\varphi \co \mathcal{Z}_P \to X(\Delta_\Sigma)/H$. The form $(\varphi^{-1})^*\omega$ on $X(\Delta_\Sigma)/H$ is what we wanted. The image of a lifted moment map is nothing but $P$ up to translations. 
	
	The construction above and Corollary \ref{coro:necessary} yields the following.
	\begin{theo}
		The holomorphic foliation $\F_{\g_J}$ on an LVMB manifold $X(\Delta_\Sigma)/H$ with indispensable integer $0$ is transverse K\"{a}hler with respect to $\F_{\g_J}$ if and only if $X(\Delta_\Sigma)/H$ is an LVM manifold with indispensable integer $0$. 
	\end{theo}
	We give remarks on the foliation $\F_{\g_J}$ and equivariant holomorphic principal bundles. Let $M_1$ and $M_2$ are complex manifolds with the complex structures $J_1$ and $J_2$, respectively. Assume that compact tori $G_1$ and $G_2$ act on $M_1$ and $M_2$ respectively. If we have an equivariant principal holomorphic bundle $\pi \co M_1 \to M_2$, it is easy to see that $T_x\F_{{\g_1}_{J_1}} = (d\pi)_x^{-1}(T_{\pi (x)}\F_{{\g_2}_{J_2}})$ for all $x \in M_1$. Therefore we can obtain every basic form for $\F_{{\g_1}_{J_1}}$ from a basic form for $\F_{{\g_2}_{J_2}}$ by the pull-back operator $\pi^*$. Moreover, every basic form for $\F_{{\g_1}_{J_1}}$ is also basic for the action of $\ker \alpha$ (that is, invariant under the action of $\ker \alpha$ and the interior product with fundamental vector fields generated by the action of $\ker \alpha$ vanishes). Therefore there exists the inverse operator $(\pi^*)^{-1}$ of $\pi^*$ defined for basic forms for $\F_{{\g_1}_{J_1}}$. 
	In particular, $\F_{{\g_1}_{J_1}}$ on $M_1$ is transverse K\"{a}hler if and only if so is $\F_{{\g_2}_{J_2}}$ on $M_2$. Also, there exists a moment map $\Phi_1 \co M_1 \to \g^*_1$ with respect to a $G_1$ invariant transverse K\"{a}hler form $\omega_1$ if and only if there exists a moment map $\Phi_2 \co M_2 \to \g^*_2$ with respect to $(\pi^*)^{-1}\omega_1$.
	
	It has been shown in \cite{Ishida} that a compact connected complex manifold $M$ equipped with a maximal action of a compact torus $G$ is obtained as a quotient of an LVMB manifold with indispensable integer $0$. Therefore, we can characterize the manifold with a maximal torus actions which admits a transverse K\"{a}hler form with respect to $\F_{\g_J}$. 
	\begin{theo}\Label{theo:polytopal}
		Let $M$, $G$, $J$, $\Delta$, $\mathfrak{h}$ be as Theorem \ref{theo:maximal}. Then, the followings are equivalent:
		\begin{enumerate}
			\item $\F_{\g_J}$ on $M$ is transverse K\"{a}hler.
			\item $q(\Delta)$ is polytopal.
		\end{enumerate}
		In this case, for any $G$-invariant transverse K\"{a}hler form $\omega$, there exists a moment map $\Phi \co M \to \g^*$ with respect to $\omega$ and the image of $M$ by a lifted moment map $\widetilde{\Phi} \co M \to (\g/\g_J)^*$ is an inner normal polytope of $q(\Delta)$. 
	\end{theo}
	As a corollary, we show that the conjecture posed in \cite{Foutou-Zaffran} holds.
	\begin{coro}
		For an LVMB manifold $M$, the holomorphic foliation $\F_{\g_J}$ is transverse K\"{a}hler if and only if $M$ is an LVM manifold. 
	\end{coro}

\end{document}